\documentclass[a4paper]{amsart}
\usepackage{amsaddr}
\usepackage[utf8]{inputenc}
\usepackage{t1enc}
\usepackage{tikz}
\usepackage{ifthen}
\usepackage{amssymb}
\usepackage{amsmath}
\usepackage{amsthm}
\usepackage{caption}
\usepackage{subcaption}
\usepackage{multirow,bigdelim}

\newtheorem{theorem}{Theorem}

\newtheorem{lemma}{Lemma}
\newtheorem{proposition}{Proposition}

\newtheorem{corollary}{Corollary}
\newtheorem{defn}{Definition}
\newtheorem{pb}{Problem}
\usepackage[colorlinks=true]{hyperref}
\usepackage{tikz-cd}

\makeatletter 
\newbox\cours@box
   {
    \nobreak\par\nobreak\noindent%
    \setbox\cours@box=\hbox{%
      \kern .5em\vbox{%
        \hrule width .7em height .7em
        \vskip\baselineskip}}%
    \wd\cours@box=0mm%
    \ht\cours@box=0mm%
    \hfill\box\cours@box%
\endlist}
\makeatother

\title{Combinatorial Search for Strong Shift Equivalence}
\author{Emmanuel Jeandel}
\address{Université de Lorraine, CNRS, Inria, LORIA, 54000, Nancy, France}
\email{emmanuel.jeandel@loria.fr}
\begin{document}

\maketitle
\begin{abstract}

  The most important open problem of symbolic dynamics consists in finding a procedure to decide whether two subshifts of finite type are   conjugate or, equivalently, whether two nonnegative integer matrices are
  strong shift equivalent.  While the problem remains hard in general, the literature also contains specific examples of matrices that  are not known to be strong shift equivalent: the Baker matrices,
  and Ashley's eight-by-eight. These specific examples have been open  for 30 years.  Using computer experiments, we show that the Baker matrices are indeed strong shift equivalent for small values of $k$,
  and that Ashley's eight-by-eight matrix is indeed strong shift   equivalent to the one-by-one matrix $\begin{pmatrix}  2 \end{pmatrix}$.
 
\end{abstract}	

The main objects of study in symbolic dynamics are sets of biinfinite words, seen as a dynamical system under the  shift action, which arise naturally when considering
symbolic coding of trajectories. For details,, see \cite{LindMarcus}.  Of particular interest are shifts of finite type, which can be
defined as the set of all binfinite paths on a finite graph. Since the field's inception, the most important question has been to decide whether two shifts of
finite type are topologically conjugate, i.e. equal up to a homeomorphism that commutes with the shift action.

The question has been studied extensively for the past 50 years, since the article of Williams \cite{williams1,williams2} which rephrased it in terms
of matrices: Two nonnegative integer matrices $M$ and $N$ are elementary strong shift equivalent if $M=RS$ and $N=SR$ for two nonnegative
(not necessarily square) integer matrices $R$ and $S$, and \emph{strong shift equivalence} is the transitive closure of elementary strong shift
equivalence.  In his article, Williams showed that topological conjugacy of shifts of finite type is equivalent to deciding strong shift equivalence.

This problem of deciding topological conjugacy has remained tantalizingly open.  It is known to be
decidable for \emph{one-sided} shifts of finite type \cite{williams1},
and is undecidable for more general shifts, or for shifts in higher
dimensions \cite{Berger2,JCSS2015}.

Some progress has been made, mostly in terms of necessary conditions \cite{parry64,parry75,ParryW,BowenFranks,krieger80,eilers12}. In particular, shift equivalence~\cite{williams1}, which is a decidable condition, has been conjectured for a long time to be equivalent to strong shift equivalence. While it has now been proven that these are distinct \cite{kim,wagoner99}, we know only of a few counterexamples.

On the other hand, it has been difficult to find sufficient conditions for two matrices to be SSE \cite{cuntzkrieger80,baker83,baker87,choe97}, the best one solving the problem only for 2 by 2 matrices with positive (or small negative) determinant.

The literature \cite{LindMarcus,Kitchens,BoyleOpen,boyle24} usually focuses on a few specific instances, which examplify the state of our (lack of) knowledge. The first open question is from Baker: Are the matrices $\begin{pmatrix} 1 & k+1 \\ k & 1 \end{pmatrix}$ and  $\begin{pmatrix} 1 & k(k+1) \\ 1 & 1 \end{pmatrix}$ strong shift equivalent? The result is known for $k = 2$ but remains open for $k \geq 3$. Another example is the Ashley matrix, defined later on in this paper, which is conjectured to be strong shift equivalent to the 1 by 1 matrix $\begin{pmatrix} 2 \end{pmatrix}$. No progress has been made on these examples since Baker's original articles, although there has been some investigation specifically on the Ashley matrix \cite{bauer92}.

In this article, we investigate the problem using computer experiments. There have been few documented instances of experiments in this subject. The first investigation was probably conducted by Baker, who spent 8 hours of computer time \cite{williams92} to prove that $\begin{pmatrix} 1 & 4 \\ 3 & 1\end{pmatrix}$ and $\begin{pmatrix} 1 & 12 \\  1 & 1\end{pmatrix}$ are SSE.
Eilers and his student Lund Jensen investigated which 2 by 2 matrices are SSE, and which invariants can prove that some are not SSE. Some account of the experiment is described in \cite{eilers12}.

Computers have now become more powerful and it is now possible to say more about these open problems. Indeed, we have solved the Baker problem for all $k \leq 10$ and for $k = 12$. We also proved that the Ashley matrix is indeed strong shift equivalent to the matrix $\begin{pmatrix} 2 \end{pmatrix}$.
The method used to find the strong shift equivalence is fairly standard, but it nevertheless requires a large amount of computational resources, particularly in terms  of memory.

An interesting byproduct of the method is that we prove a generalization of the cases $k = 2$ and $k = 3$ which suggest that a hidden structure is to be found (see section \ref{thm:general}).

The article is organized as follows. We first give the necessary definitions to understand and formulate the problem, and some links to well-known results that will be necessary to understand the discussions. We then state the theorems, with all proofs relegated to the appendices, and then give a brief overview of how the computer program works.

\section*{Acknowledgements}

Experiments presented in this paper were carried out using the Abaca infrastructure, supported by Inria (see \url{https://abaca.inria.fr}).
E.J. wants to thanks Mike Boyle, Nishant Chandgotia, Brian Marcus for discussions on the paper, as well as Søren Eilers for providing the results of his previous experiment.

\section{Background}

\subsection{Definitions}
Wherever possible, we  follow the notations from \cite{boyle24}.
Let $\mathcal{S}$ be a semiring, usually $\mathcal{S} = \mathbb{Z}_+$, the set of nonnegative integers.
 
\begin{defn}
  Two square matrices $A$ and $B$, possibly of different size, with elements in $\mathcal{S}$, are elementary strong shift equivalent (ESSE, or ESSE-$\mathcal{S}$ if we want to emphasize the semiring), if there exist matrices $R,S$ over $\mathcal{S}$ s.t., $A = RS$ and $B = SR$.

  Strong shift equivalence is the symmetric, reflexive, and transitive closure of elementary strong shift equivalence:  Two matrices $A$ and $B$ over $\mathcal{S}$ are strong shift equivalent (SSE, or SSE-$\mathcal{S}$) if there exist matrices $A_1, \dots, A_l$ s.t. $A = A_1$, $A_l = B$, and $A_i$ is ESSE-$\mathcal{S}$ to $A_{i+1}$.
\end{defn}

\begin{pb}[Main open problem of symbolic dynamics]
Can we decide SSE equivalence over $\mathbb{Z}_+$ ?  
\end{pb}
As explained above, by \cite{williams1}, this problem is equivalent  to deciding whether two subshifts of finite type are conjugate.

\subsection{Necessary conditions}

The literature abounds with necessary conditions for SSE, usually formulated in terms of invariants: an invariant $\phi$ assigns to each matrix $M$ a quantity  $\phi(M)$ s.t. if $M$ and $N$ are SSE, then $\phi(M) = \phi(N)$.

We refer the reader to \cite{LindMarcus,boyle24} for a comprehensive list of necessary conditions and state the following: matrices that are SSE have the same trace and have the same nonzero spectrum (i.e., the list of eigenvalues with their multiplicities, excluding $0$) \cite{ParryW}.

Furthermore, matrices that are SSE are also \emph{shift equivalent}:

\begin{defn}[\cite{williams70,williams1}]
  \label{defn:se}
  Two nonnegative integer matrices  $A$ and $B$ are shift equivalent (SE) if there exist two nonnegative integer matrices $R$ and $S$ and an integer $k$ s.t.
  \begin{itemize}
    \item $AR = RB$,
    \item $SA = BS$,
    \item $A^k = RS$,
    \item $B^k = SR$.
  \end{itemize}
  The integer $k$ is called the \emph{lag} of the shift equivalence.
\end{defn}

Given an SSE $A = M_1 = R_1S_1, S_1R_1=M_2, \dots, S_nR_n = M_n = B$, one obtains a shift equivalence by taking $R = R_1 \cdots R_n$ and $S = S_n \cdots S_1$.
The article \cite{williams1} mistakenly states that there is a converse \cite{williams2}; however counterexamples were later found \cite{kim,wagoner99}. The fact that there are so few known counterexamples means, however, that testing for shift equivalence is usually a good indication that matrices are SSE, and the lag of the shift equivalence is a lower bound for the number of steps in the SSE, if it exists.

Similarity over $\mathbb{Z}$ is another good indication that matrices may be SSE. Indeed, if $A$ and $B$ are nonnegative integer primitive matrices\footnote{A nonnegative matrix is primitive if there exists an integer $n$ s.t. all entries of $A^n$ are positive.} s.t. $B = PAP^{-1}$ for some $P \in GL_n(\mathbb{Z})$, then $A$ and $B$ are shift equivalent \cite[Theorem 7.3.6]{ParryW,LindMarcus}. See \cite{ha96} for other links between SSE and similarity.


\subsection{Sufficient conditions}

Few good sufficient conditions have been found; most of them focus on $2\times 2$ matrices.

\begin{theorem}
  \label{thm:conditions}
  In the following, all matrices are $2 \times 2$ primitive nonnegative integer matrices.
  \begin{itemize}
  \item \cite{cuntzkrieger80} If $A$ and $B$ have determinant $1$, then they are SSE  iff they have isomorphic dimension groups and the same entropy (in particular iff they are shift equivalent).
  \item \cite{baker87} If $A$ and $B$ are similar over $\mathbb{Z}$ and $\det A \geq -\mathrm{tr}(A)$, then they are SSE.
  \item \cite{choe97} If $A$ and $B$ are similar over $\mathbb{Z}$, $\det A \geq -2\mathrm{tr}(A)$, and $|\det A|$ is not a prime, then they are SSE.    
  \end{itemize}  
\end{theorem}

The method explained in \cite{baker87} is slightly more general, as  it can sometimes be applied to higher dimensional matrices.

\subsection{Open problems}

Little is known on the $2 \times 2$ case beyond the previous theorem. In particular, the following problem remains open:
\begin{pb}
  If $A$ and $B$ are $2 \times 2$ primitive nonnegative integer matrices similar over $\mathbb{Z}$, are they SSE ?  
\end{pb}
This is known to be false for $7 \times 7$ matrices, as shown in the seminal article of Kim and Roush \cite{kim}.

One concrete instance of the problem is given by the examples of Baker:
\begin{pb}[Baker's matrices \cite{LindMarcus,BoyleOpen,boyle24}]
  Are the matrices $A_k = \begin{pmatrix} 1 & k+1 \\ k & 1 \end{pmatrix}$ and  $B_k = \begin{pmatrix} 1 & k(k+1) \\ 1 & 1 \end{pmatrix}$ SSE ?
\end{pb}
The result is  known only for $k \leq 2$.

Another way to instantiate the general open problem of deciding SSE in general is to fix the matrix $B$:

\begin{pb}[Little shift equivalence conjecture for $n=2$ \cite{Kitchens,BoyleOpen,boyle24}]
  If $A$ is a nonnegative integer matrix with a single nonzero eigenvalue $2$ of multiplicity 1, is $A$ SSE to the matrix $\begin{pmatrix} 2 \end{pmatrix}$ ?
\end{pb}
A concrete open example of the little shift equivalence conjecture is Ashley's eight-by-eight which, up to permutation, is given by the following matrix:

\[
\begin{pmatrix}
  1&1&0&0&0&0&0&0\\
  0&0&1&0&0&1&0&0\\
  0&0&0&1&0&0&1&0\\
  0&1&0&0&1&0&0&0\\
  0&0&0&1&0&1&0&0\\
  0&0&1&0&0&0&1&0\\
  0&0&0&0&1&0&0&1\\
  1&0&0&0&0&0&0&1
  \end{pmatrix}
\]

This matrix was studied in detail in  \cite{bauer92} in an attempt do prove  it is not SSE to $\begin{pmatrix} 2 \end{pmatrix}$.

\section{Results}

In this section, we summarize the results we obtained. Due to the nature of the proof, and to not submit a 200 pages articles, the proofs in the article but on the accompanying website \url{https://members.loria.fr/EJeandel/sse.html} (This will be switched for a more permanent solution for the final version of the article).

\subsection{Baker matrices}
First, we focus on Baker's matrices, and more generally on $2\times 2$ matrices.
Through computer experiments, we prove the following:
\begin{theorem}
\label{thm:baker}
  Baker's  matrices $A_k$ and $B_k$ are SSE for $k \leq 10$ and $k = 12$. The SSE uses $6\times 6$ matrices for $k =8, 9, 10$ and $5 \times 5$ matrices otherwise.

  The number of steps of the SSE is:
  
  \begin{tabular}{|l|l||l|l|}
    \hline
    $k$ & Steps & $k$ & Steps\\
    \hline
    3 & 17& 8 & 110 \\
    4 & 28& 9 & 154\\
    5 & 44& 10 & 193\\
    6 & 73& 12 & 14\\
    7 & 17&&\\
    \hline
    \end{tabular}
  \end{theorem}
Our experiments suggest that the number of steps of an SSE from $A_k$ to $B_k$, assuming it exists, grows quadratically with $k$.
As $k$ increases, our algorithm requires more memory. It is thus not surprising that we we could not find any SSE for any $k \geq 11$.
Surprisingly, there is an exception for $k = 12$. The SSE for $k = 7$ is also surprisingly small.

We cannot yet provide an explanation, but we offer the following insight. From an SSE from $A_k$ to $B_k$, one can construct a shift equivalence (SE) from $A_k$ to $B_k$ using the process described after Definition~\ref{defn:se}.
For all  $k$ except $7$ and $12$, the shift equivalence we obtain is the obvious one arising from the similarity of $A_k$ and $B_k$, namely $R = \begin{pmatrix} 1 & k+1 \\ 1 &k \end{pmatrix}$ and $S = R^{-1} A_k^n$ where $n$ is chosen large enough for $S$ to be nonnegative.
In the remaining cases ($k=7$ and $k=12$), We obtain different shift equivalences. For $k = 7$, we obtain a SE of lag $4$ with $R = \begin{pmatrix}1 & 16 \\ 2 & 7\end{pmatrix}$ and $S = R^{-1} A_7^4$.
For $k = 12$, we obtain a SE  of lag $4$ with $R = \begin{pmatrix}1 & 91 \\ 7 & 12\end{pmatrix}$ and $S = R^{-1} A_{12}^4$.
Preliminary work suggests, but does not prove, that these are the only values of $k > 2$ for which a SE of lag $4$ exists.

\subsection{Generalizations}

Once a proof of SSE from, say, $A_3$ to $B_3$ is obtained,  the intermediate steps can be adjusted to yield different SSE chains between the same matrices.
it is therefore not clear which proof is the \emph{right one}, and if in fact the computer exploration found one that can be explained and generalized.

The next theorem provides, for $k = 2,3,4$, proofs that are generic enough to be extended to numerous other SSEs:

\begin{theorem}
  \label{thm:general}

  Let $S = \mathbb{Z}_+[x,y,z]$ be the semiring of nonnegative polynomials in three variables.

Define
$C_k = \begin{pmatrix}  x+z & x+(k+1)z\\  y+kz & y+z\\\end{pmatrix}$
and
$D_k = \begin{pmatrix}
  x+z & (x+(k+1)z)(y+kz)\\
  1 & y+z\\
\end{pmatrix}$.

For $k = 2,3,4$, $C_k$ and $D_k$ are SSE in the semiring $S$, respectively in 7, 17 and 48 steps.
\end{theorem}

In more intuitive terms, for all nonnegative integer values of $x, y, z$, there exists a SSE from
$\begin{pmatrix}  x+z & x+3z\\  y+2z & y+z\\\end{pmatrix}$
to
$\begin{pmatrix}
  x+z & (x+3z)(y+2z)\\
  1 & y+z\\
\end{pmatrix}$

Moreover, this SSE is independent of the exact values of $x,y$ and $z$.
For $k=2,3$, the number of steps coincides with the result in Theorem~\ref{thm:baker}. This is not the case for $k = 4$. While $48$ might not be optimal, it is not possible to transform the proof we found for $A_4$ and $B_4$ into a proof for $C_4$ and $D_4$.

The Baker matrices correspond to the case where $x=y=0$, $z=1$.
The case $x=1,y=0,z=1$ gives a SSE from $\begin{pmatrix}
  2 & 5 \\
  3 & 1
\end{pmatrix}$
to
$\begin{pmatrix}
  2 & 15 \\
  1 & 1
\end{pmatrix}$, which was posed as an open problem in \cite{baker83}.

As an interesting byproduct of the theorem, consider the sequences of matrices $G_k = \begin{pmatrix} k & 5 \\ 4 & k \end{pmatrix}$ and $H_k = \begin{pmatrix} k & 20 \\ 1 & k \end{pmatrix}$.
The matrices $G_k$ and $H_k$ are similar over $\mathbb{Z}$ via the same two by two matrix $P$ for all $k$.
By Theorem~\ref{thm:baker}, $G_1$ and $H_1$ are SSE using 28 steps. Moreoever $G_2$ and $H_2$ are SSE using 17 steps since $G_2 = C_3$ for $x=y=z=1$. Additionally $G_3$ and $H_3$ are SSE using 7 steps as $G_2 = C_2$ for $x=y=2$ and $z=1$.
$G_4$ and $H_4$ are easily seen to be elementary SSE, and the same holds for $G_5$ and $H_5$. For $k \geq 6$, the matrices  have positive determinant, and therefore are SSE by Theorem~\ref{thm:conditions}.

This suggests that, when two matrices $A$ and $B$ are similar over $\mathbb{Z}$, the matrix $P$ realizing the similarity does not reliably predict the complexity of the SSE.

\subsection{Other 2 by 2 matrices}

\begin{theorem}
  The matrices $E_n = \begin{pmatrix} 1 & n \\ 2 & 1 \end{pmatrix} $ and $F_n = \begin{pmatrix} 1 & 2n \\ 1 & 1\end{pmatrix}$
    are SSE for all odd  $n < 50$ except $n = 15, 21, 37, 45$. For these four values, the matrices are not even shift equivalent, and hence not SSE.
\end{theorem}
It is easily shown that the matrices are not SSE if $n$ is even using e.g. the Bowen Franks group \cite{BowenFranks}. For  $2\times 2$ matrices, the Bowen Franks group invariant can be restated as follows: if $A$ and $B$ are SSE then $gcd(A-I) = gcd(B-I)$ where $I$ is the identity matrix and $gcd(X)$ denotes the greatest common divisors  of all entries in $X$.

The case $n = 5$ was posed in an equivalent form as an open problem in \cite{baker87}.

The proof that some matrices are not shift equivalent is a direct consequence of the following lemma:
\begin{lemma}
  Let $n$ and $p$ be integers.  
  Let $A = \begin{pmatrix} 1 & n \\ p & 1 \end{pmatrix} $ and $B = \begin{pmatrix} 1 & np \\ 1 & 1\end{pmatrix}$ .
    If $np - 1$ is prime and neither $p$ nor $-p$ is a square modulo $n$, then $A$ is not shift equivalent to $B$.
\end{lemma}
\begin{proof}
Suppose there exists matrices $R,S$ and an integer $k$ s.t. $AR = RB$, $SB = AS$, and $RS = A^k$, $SR = B^k$.

  A straightforward computation shows that $R = \begin{pmatrix} a & nb \\ b & pa \end{pmatrix}$   and $S = \begin{pmatrix} pc & nd \\ d & c \end{pmatrix}$.
  From $RS = A^k$ we get by looking at the determinant that $(pa^2 - n b^2)(pc^2 - nd^2) = (1-np)^k$.
  Since $np-1$ is prime this implies that $pa^2 - nb^2 =  \pm (1-np)^{k'}$ for some integer $k'$ and therefore that $pa^2 \equiv \pm 1 \mod n$, which means that either $p$ or $-p$ is a square modulo $n$.
\end{proof}

\subsection{The Ashley example}

Our algorithm can be used with other matrices, and we were also able to solve Ashley eight-by-eight problem:
\begin{theorem}
  The Ashley  matrix is SSE to the matrix $\begin{pmatrix}2 \end{pmatrix}$. The SSE uses 23 steps and involves 10 by 10 matrices.
  \end{theorem}

\section{A Bird's Eye view of the Algorithm}

\subsection{Finding a strong shift equivalence}
\label{proof:bird}
In this section, we describe the algorithm used to obtain our results. The algorithm follows a standard approach: Starting from a  matrix $A$ and a target matrix $B$, we iteratively compute all matrices that are SSE in $k$ steps to $A$, until either $B$ is found or we run out of memory. However, several challenges arise when implementing this approach directly.

First of all, no efficient algorithm is known, given a matrix $M$, to compute all matrices $R,S$ s.t $M = RS$, and therefore no efficient algorithm exists to compute all matrices that are elementary SSE to a given matrix.
Instead, at each step, given a matrix $M$, we apply the following \emph{moves} (using the terminology of \cite{eilers25}):

\begin{itemize}
\item(row splitting) Split the $i$-th row $u$ of $M$ into two rows $v_1, v_2$ s.t. $u = v_1 + v_2$, and duplicate the $i$-th column
\item(column splitting) Split the $i$-th column $u^T$ of $M$ into two columns $v_1^T, v_2^T$ s.t. $u = v_1 + v_2$, and duplicate the $i$-th row
\item(almagamation) If columns $i$ and $j$ are equal, remove one, and add the corresponding rows (this is the reverse of row splitting)
\item(almagamation) If rows $i$ and $j$ are equal, remove one, and add the corresponding columns (this is the reverse of column splitting)
\item(shear move) Subtract the $i$-th row of $M$ from the $j$-th row (if the result remains nonnegative) and add the $j$-th column to the $i$-th column
\item(shear move) Subtract the $i$-th column of $M$ from the $j$-th column (if the result remains nonnegative) and add the $j$-th row to the $i$-th row
\item(scale move) Divide the $i$-th row of $M$ by $k$ (if the result remains integral) and multiply the $i$-th column by $k$
\item(scale move) Divide the $i$-th column of $M$ by $k$ (if the result remains integral) and multiply the $i$-th row by $k$
\end{itemize}

All matrices $N$  obtained from $M$ via these moves are 
elementary SSE to $M$, though the converse does not hold. Nevertheless, these moves suffice to find a SSE from $M$ to $N$ if one exists (in fact, splits and amalgamations alone are sufficient \cite{williams1}).

"The primary challenge in implementation is the sheer number of matrices SSE to a given matrix. To address this, we adopt the following strategies:
\begin{itemize}
\item Matrices are considered up to permutation: two matrices $M$ and $N$ are identified if $M = PNP^{-1}$ for a permutation matrix $P$.  This requires computing a canonical representative for each equivalence class.
\item We restrict our consideration to \emph{essential} matrices, i.e., matrices with no zero rows or columns.
\item We bound the size of matrices under consideration, either by limiting their coefficients or their dimensions. Additionally, matrices are stored in a compressed format.
\end{itemize}

The number of matrices to process remains tantalizingly large. For example, we found approximately 72 million $5 \times 5$ or smaller matrices that are SSE to the matrix $A_3$.
To find the SSE from $A_9$ to $B_9$, we had to consider 23 billion matrices of size up to $6 \times 6$, which represents only a tiny fraction of all $6 \times 6$ matrices SSE to $A_9$.

For small $k$, we can visualize the algorithm’s behavior. Figure~\ref{graph:toto} shows all matrices of size at most $4 \times 4$ (up to permutation) reachable from $A_2$ by the algorithm:
The graph has an edge from matrix $M$ to matrix $N$ if $N$ can be obtained from $M$ by applying one of the moves listed above (since the moves are symmetric, the graph is undirected). The graph contains 570k nodes and 640k edges.
The graph exhibits an obvious symmetry corresponding to the transpose operation. Since $A_2$ is, up to permutation, equal to its transpose, the set of reachable matrices is closed under transposition. Moreover, if there is an edge from $M$ to $N$, then there is also an edge from $M^T$ to $N^T$.

More surprisingly, the graph’s structure clearly reveals two distinct clusters, with $A_2$  in one of the clusters, and $B_2$ in the other cluster.
Furthermore, there are relatively few paths between the clusters, and it is easy to see that the two clusters can be disconnected by removing just two vertices. We identified those vertices: they are the matrix
  \[  \begin{pmatrix}   0& 1& 0& 1\\    1& 0& 1& 1\\    0& 0& 1& 2\\    2& 1& 1& 1\\    \end{pmatrix}\]
  and its transpose. 

We observe a similar pattern for $A_3$ and $B_3$: The source matrix $A_3$ and the target matrix $B_3$ are in two separate clusters connected by a narrow set of intermediate matrices.
Our experiments suggest this pattern may also hold for larger $k$, though we were unable to visualize the full graph to confirm it.

\begin{figure}
  \includegraphics[width=10cm]{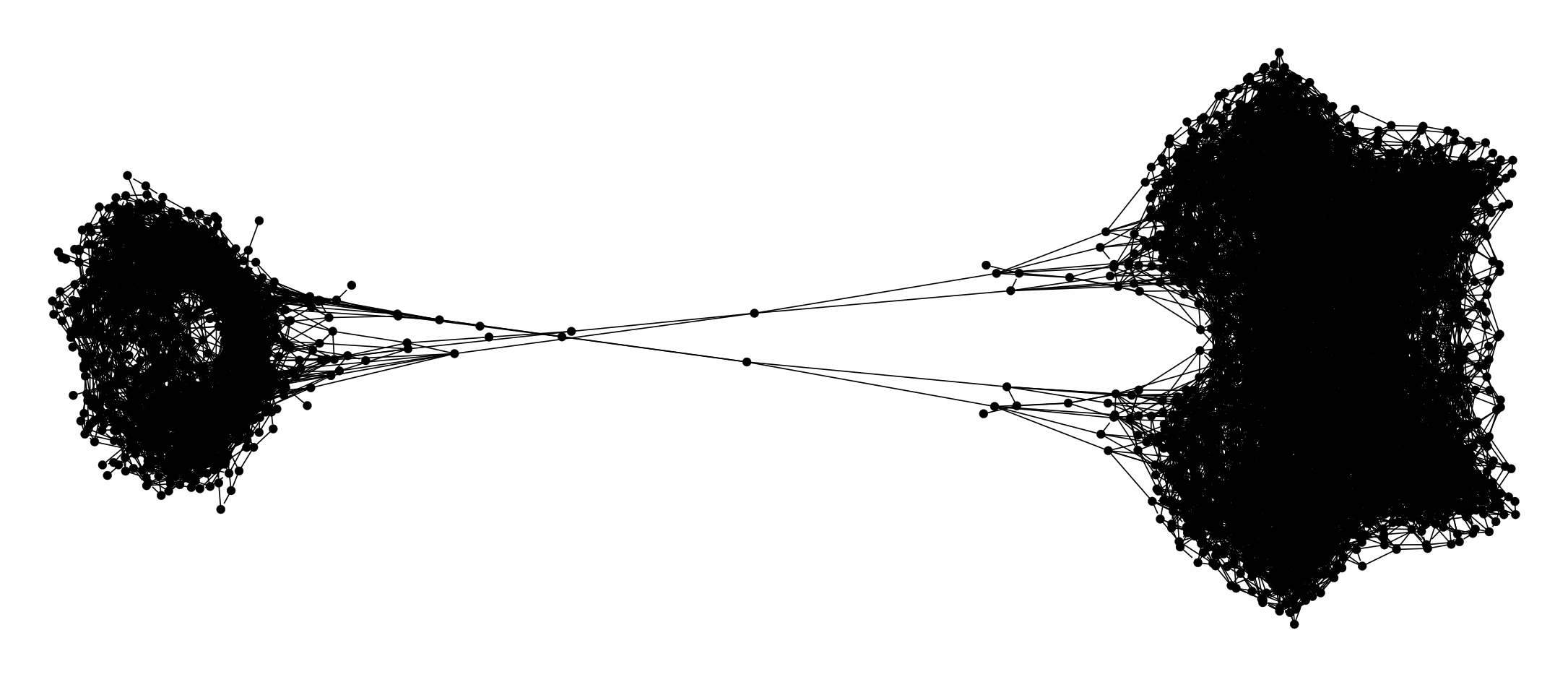}
  \caption{The graph of all 4 by 4 matrices that are SSE to the matrix $A_2$ found by the algorithm}
  \label{graph:toto}
\end{figure}

\subsection{Optimizing strong shift equivalences}

We now briefly explain how to optimize a strong shift equivalence between two matrices to obtain a shorter path.

Consider a strong shift equivalence, and in particular, two consecutive steps:
\begin{itemize}
\item $A_k = R_k S_k$
\item $S_k R_k = A_{k+1}$
\item $A_{k+1} = R_{k+1} S_{k+1}$
\item $S_{k+1} R_{k+1} = A_{k+2}$
\end{itemize}
Suppose there exists a nonnegative integer matrix $M$ with full column rank that left-factors both $S_k$ and $R_{k+1}$, i.e., $S_k = M S'_k$ and $R_{k+1} = M R'_{k+1}$.

We can then construct an alternative strong shift equivalence from $A_k$ to $A_{k+2}$ as follows:
\begin{itemize}
\item $A_k = (R_k M) S'_k$
\item $S'_k (R_k M) = A'_{k+1}$
\item $A'_{k+1} = R'_{k+1} (S_{k+1} M)$
\item $(S_{k+1} M) R'_{k+1} = A_{k+2}$
\end{itemize}
We must verify that the two expressions for $A_{k+1}$ agree. In fact $S'_k R_k = R'_{k+1} S_{k+1}$. This holds because $M (S'_k R_k) = S_k R_k = R_{k+1} S_{k+1} = M (R'_{k+1} S_{k+1})$ and $M$ has full column rank.

This new strong shift equivalence might be simpler than the original: If $M$ is nonsquare, $A'_{k+1}$ has  smaller dimension. Additionally,  $A'_{k+1}$ may become the identity matrix, in which case one can bypass it and go directly from $A_k$ to $A_{k+2}$ in one step.
Our optimization algorithm, given a SSE path from $A$ to $B$, iteratively generates all SSE paths obtainable via this method and attempts to find a shorter path.

This can be explained using Wagoner's concept of the $SSE(\mathbb{Z}_+)$ complex~\cite{wagonerRS,wagoner92}:
\begin{itemize}
\item \emph{Vertices} are square matrices
\item \emph{Edges} are elementary SSE: There is a directed edge from $A$ to $B$ labeled $(R,S)$ if $RS = A$, $SR = B$.
\item \emph{Triangles} are diagrams of the form:
\[
\begin{tikzcd}[column sep=large, row sep=large]
& C  \\
A \arrow[ur, "{(R_3, S_3)}"] \arrow[r, "{(R_1, S_1)}"'] & B \arrow[u, "{(R_2, S_2)}"'] 
\end{tikzcd}
\]
where $R_1R_2 = R_3$, $S_1 = R_2 S_3$ and $S_2 = S_3 R_1$
\end{itemize}

It is straightforward to see that our transformation corresponds to two triangles in the complex:
\[
\begin{tikzcd}[column sep=huge, row sep=large]
& A'_{k+1}   \arrow[dr, "{(R'_{k+1}, S_{k+1}M)}"] \\
A_k \arrow[ur, "{(R_kM, S'_k)}"] \arrow[r, "{(R_k, M S'_k)}"'] & A_{k+1} \arrow[u, "{(M, S'_k R_k)}"'] \arrow[r, "{(M R'_{k+1}, S_{k+1})}"'] & A_{k+2}
\end{tikzcd}
\]

In particular, our optimization strategy replaces the original path with a new path homotopic to it in this complex.
For matrices and paths in $\{0,1\}$, homotopic paths represent the same conjugacy of
the two subshifts of finite type.  In $\mathbb{Z}_+$, they represent
the same conjugacy up to simple automorphisms
\cite{wagonerRS,wagoner92}. Thus it is  intuitively reasonable to
consider the original path and the optimized path as representing the
same conjugacy. It would be interesting to develop criteria for distinguishing non-homotopic SSE paths, thereby determining whether two computer-generated proofs are fundamentally different or merely distinct representations of the same proof.

\section{Conclusion}

Many questions remain following the experiment. First, the SSE we found for the Baker matrices do not provide a clear path to generalization for arbitrary $k$, and thus the problem remains open for large $k$.
Another open question is whether the generalizations in Theorem~\ref{thm:general} can be extended whenever a strong shift equivalence exists between two $2 \times 2$ matrices, or if additional conditions are required.

Finally, in all our experiments, whenever we found a SSE (using the moves in Section~\ref{proof:bird}) between $2 \times 2$ matrices, this SSE required only intermediate matrices of corank 1. This leads us to conjecture that an SSE between matrices of corank $k$ requires only intermediate matrices of corank $k+1$.

\appendix

\section{Optimality of the algorithm}

As we explained before, the algorithm does not look at each step at each possible elementary strong shift equivalence, but only at a specific number of moves. The natural question is therefore: if the algorithm finds a SSE from $M$ to $N$ with $k$ moves involving only $p$ by $p$ matrices, what can we say about the length of the shortest SSE, and the size of the matrices ? We answer this question in this appendix.

\subsection{Moves}

A \emph{move} is a relation $\mathcal{R}$ on the set of all square nonnegative integer matrices with the property that $M \mathcal{R} N$ implies that $M$ and $N$ are SSE.
If $M \mathcal{R} N$, we say that $M$ is transformed into $N$ by the move $\mathcal{R}$.

A set of moves $S$ is \emph{complete} if, whenever $M$ and $N$ are SSE, one can tranform $M$ into $N$ using only moves in $S$.
A set of moves is \emph{symmetric} if, whenever there is a move in $S$ that transforms $M$ into $N$, there is one that transforms $N$ into $M$.

We will only be interested in what follows  in sets of moves that are complete and symmetric.

\subsection{The list of moves}
\begin{defn}
  An elementary \emph{division matrix} is a $\{0,1\}$ matrix $D$ of size $n-1\times n$ s.t. each row contains at least one $1$, and each column contains exactly one $1$.
  
An elementary \emph{shear matrix} is a matrix $E$ of size $n\times n$ which is equal to the identity except it has one nondiagonal  coefficient with value $1$.
If the nondiagonal coefficient is larger than $1$, the matrix is called an extended shear matrix.

An elementary \emph{scaling matrix} is a matrix $S$ of size $n\times n$ which is equal to the identity except one diagonal coefficient is larger than $1$.
\end{defn}

Division matrices were introduced in Williams \cite{williams1}, requiring that each row has exactly one symbol $1$. This conflicts with the definition in Lind and Marcus \cite{LindMarcus} that requires that each column has exactly one symbol $1$. We take here the second definition, and we add the prerequisite (hence the word ``elementary'') that the difference between the number of rows and the number of columns is exactly~1. Shear matrices and scaling matrices are well known in linear algebra and are sometimes called ``elementary matrices''. We use the vocabulary ``shear'' from Baker \cite{baker87}.

\begin{defn}
  We say that $N$ is an elementary split of $M$, and $M$ is an elementary amalgamation of $N$, if, up to permutation, there exists a matrix $A$  and an elementary  division matrix $D$ s.t. either $M = DA$ and $N = AD$ (row split), or $M = AD^T$ and $N = D^TA$ (column split). The split move is the relation $\mathcal{S}$ on essential matrices s.t. $M\,\mathcal{S}\,N$ if $N$ is a split of $M$ or $N$ is an amalgamation of $M$.
  
 We say that $N$ is obtained from $M$ by a shear if there exists a matrix $A$  and an elementary shear matrix $E$ s.t. either $M = EA$ and $N = AE$ , or $M = AE$ and $N = EA$. The shear move is the relation $\mathcal{E}$ on essential matrices s.t. $M\,\mathcal{E}\,N$ if $N$ is obtained from $M$ by shear. We define the extended shear move, noted $\mathcal{E}_{ext}$ in a similar way.       

 We say that $N$ is obtained from $M$ by a rescaling if there exists a matrix $A$  and an elementary scaling matrix $P$ s.t. either $M = PA$ and $N = AP$ , or $M = AP$ and $N = PA$. The product move is the relation $\mathcal{P}$ on essential matrices s.t. $M\,\mathcal{P}\,N$ if $N$ is obtained from $M$ by shear.
\end{defn}

All of these moves can be defined in terms of rows and columns operations:
\begin{itemize}
\item  A row split consists in splitting the $i$-th row $u$ into two nonzero rows $v,v'$ s.t. $u = v+v'$ and duplicating the $i$-th column. The column split is defined accordingly.
\item  A shear move consists in subtracting the $i$-th row from the $j$-th row, assuming the result is nonnegative and nonidentically zero, and then adding the $j$-th column to the $i$-th column, or doing the reverse.
An extended shear move consists in subtracting $k$-th times the $i$-th row from the $j$-th row.
\item  A rescaling move consists in dividing the $i$-th row by a number $p$ (if the result is integral) then multiplying the $i$-th column by the number $p$, or doing the reverse.
\end{itemize}

\begin{proposition}\label{prop:moves}
  \begin{itemize}
  \item If $M$ and $N$ are SSE by a sequence of matrices of size at most $n$, one can move  from $M$ to $N$ by split moves involving only matrices of size at most $2n+1$.
  \item If $M$ and $N$ are SSE by a sequence of matrices of size at most $n$, one can move  from $M$ to $N$ by split moves and shear moves involving only matrices of size at most $2n$.
  \item If $M$ and $N$ are SSE by a sequence of matrices of size at most $n$, one can move  from $M$ to $N$ by split moves, shear moves and rescaling moves, involving only matrices of size at most $2n-1$.
  \item If $M$ and $N$ are SSE by a sequence of $k$ matrices of size at most $n$, one can move  from $M$ to $N$ by a sequence of at most $O(n^2k)$ split moves, extended shear moves, and rescaling moves, involving only matrices of size at most $2n-1$.
  \end{itemize}    
\end{proposition}

In other words, if the algorithm finds a SSE using split moves, shear moves and rescaling moves from $M$ to $N$ with matrices of size $n$ by $n$, but not with $n-1$ by $n-1$ matrices, then we know the SSE with the smallest possible matrices needs matrices with at least $\frac{n+1}{2}$ rows and columns.

The proposition is a consequence of the following lemma:
\begin{lemma}
  \label{lemma:product}
  \begin{enumerate}
    \item\label{lemma:case:i} Every essential rectangular matrix $R$ with nonnegative integer coefficients is the product of elementary division matrices, their transpose, and of one permutation matrix. More precisely, if $R$ is a $n\times m$ matrix, then each matrix in the product has size bounded by $n+m+1$.   
    \item\label{lemma:case:ii} $R$ can be written as a product involving shear matrices as well as the previous classes of matrices. In this case, each matrix in the product has size bounded by $n+m$.
    \item\label{lemma:case:iii} $R$ can be written as a product involving scaling matrices as well as the previous classes of matrices. In this case, each matrix in the product has size bounded by $n+m-1$.
    \item\label{lemma:case:iv} $R$ can be written as a product involving extended shear matrices as well as the previous classes of matrices. In this case, each matrix in the product has size bounded by $n+m-1$, and the number of matrices in the product is bounded by $2m(n+m+2)$.
      \end{enumerate}
\end{lemma}

\begin{proof}[Proof of Proposition \ref{prop:moves}]
  Let $(x)$ denote one of the 4 cases of the proposition.  It is enough to prove case $(x)$ when $M$ and $N$ are SSE in one step, i.e.  $M = RS$ and $N = SR$.

Write $R = R_1 \dots R_n$ using case $(x)$ of the lemma. Then we obtain a SSE from $M$ to $N$ of length $n$, starting from $M = R_1 \dots R_nS$, then going through  $R_2 \dots R_n S R_1$, $R_3\dots R_n S_1 S_2$ up to $S R_1 \dots R_n = N$. By construction, this SSE only uses essential matrices and all moves are allowed in case $(x)$ with matrices of the given size.  
  \end{proof}

\begin{proof}[Proof of Lemma \ref{lemma:product}]
  We prove case $(\ref{lemma:case:ii})$, and then give some details on how to prove the other cases.

If $M = DN$ where $D$ is a division matrix, then $N$ is obtained, up to a permutation, from $M$ by splitting a row $u$ into two rows $u_1, u_2$ of sum equal to $u$.
If $M = D^T N$ where $D$ is a division matrix, then $N$ is obtained, up to a permutation, from $M$ by removing a duplicate row.
If $M = EN$ for $E$ a shear matrix, then $N$ is obtained from $N$ by subtracting a row from another, verifying the remaining row is nonnegative and nonidentically zero.

Therefore it is sufficient to prove the following result: We prove by recurrence on $m$ that any essential matrix $R$ with $n$ rows and $m$ columns can be obtained from the identity matrix by permuting rows, splitting rows, removing duplicate rows, and subtracting rows, with each matrix in the process having less than $n+m$ rows.  The proof is essentially Gaussian elimination, being careful during the proof that the matrices $R$ remain essential.

If $m = 1$, $R$ is a row vector, with each coefficients being non zero.
By splitting the first row if necessary, we may suppose that one of the coefficients of $R$ is $1$. We can then subtract this coefficient from each other coefficient until they are all equal to 1, and then remove the duplicate rows until we obtain the matrix $\begin{pmatrix} 1 \end{pmatrix}$. 

If $m > 1$, as $R$ is essential, there is at least one row that contains a nonzero coefficient in its first column. By splitting this row if necessary, we may suppose wlog that the row is of the form $\begin{pmatrix} 1 & 0 \cdots & 0 \end{pmatrix}$. We then subtract this row from all others until all rows are either of the form $\begin{pmatrix} 1 & 0 & \cdots & 0\end{pmatrix}$ or of the form $\begin{pmatrix} 0 & x \end{pmatrix}$ where $x$ is a nonzero vector.
  We then remove all duplicate rows of the form $\begin{pmatrix} 1 & 0 & \cdots & 0\end{pmatrix}$.

    Doing all this, we have transformed our matrix $R$ into a matrix of the form $\begin{pmatrix} 1 & 0 \\ 0 & R' \end{pmatrix}$, with $R'$ essential with at most $n$ rows and $m-1$ columns, the whole process having used intermediate matrices of size at most $n+1 \leq m+n$.

    By recurrence, $R'$ can be transformed into the identity matrix using intermediate matrices of size at most $m+n-1$. Therefore  $\begin{pmatrix} 1 & 0 \\ 0 & R' \end{pmatrix}$ can be transformed into the identity matrix using intermediate matrices of size at most $m+n$, which ends the proof.

    Case $(\ref{lemma:case:i})$ is similar, noting that we can simulate subtracting row $u$ from row $v$ by first splitting row $u$ into $v$ and $u-v$ and then removing the duplicate row $v$. This will add temporarily another row to the matrix.

    Case $(\ref{lemma:case:iii})$ only differs in the handling of the case $m = 1$, where we can use the scaling matrix to suppose that one of the coefficient is $1$ instead of splitting the row.

    Case $(\ref{lemma:case:iv})$ is similar to case $(\ref{lemma:case:iii})$ looking carefully at the number of operations. We need extended shears to be able to convert a coefficient $p$ to $0$ in one step instead of doing it in  $p$ steps.
    Each step of the recurrence uses $2n+1$ operations to go from a matrix with $n$ rows and $m$ columns to a matrix with $n+1$ rows and $m-1$ columns, which gives a total number of steps less than $\sum_{k=n}^{n+m-2} {2k+1}$.
\end{proof}

\clearpage

\section{$A_3$ is SSE to $B_3$ (unoptimized)}

This proof only use the moves provided in section~\ref{proof:bird}. Proofs in the later section are optimized for number of steps, but might be harder to analyze

\[ A_{0} = 

= A_{32} \]

\clearpage

\section{$A_3$ is SSE to $B_3$ (optimized)}

\[ A_{0} = 
%
= A_{17} \]

\clearpage

\section{Ashley eight-by-eight is SSE to the matrix $(2)$ (unoptimized)}

{\setlength{\arraycolsep}{2pt}
\[ A_{0} = 
%
= A_{47} \]

}
\section{Ashley eight-by-eight is SSE to the matrix $(2)$ (optimized)}

{\setlength{\arraycolsep}{2pt}
  \[ A_{0} = 
%
= A_{23} \]

  }

\end{document}